\documentclass{scrartcl}

\usepackage{amssymb}
\usepackage{amsmath}
\usepackage[all]{xy}

\usepackage{amsthm}

\newtheorem{theorem}{Theorem}[section]
\newtheorem{corollary}[theorem]{Corollary}
\newtheorem{proposition}[theorem]{Proposition}
\newtheorem{lemma}[theorem]{Lemma}

\theoremstyle{definition}
\newtheorem{definition}[theorem]{Definition}

\theoremstyle{remark}
\newtheorem*{remark}{Remark}

\begin{document}

\title{Abelian varieties with quaternion and complex multiplication}

\author{Dominik Ufer}

\date{}

\maketitle

\begin{abstract}
In this paper we study abelian varieties $A$ which correspond to CM points
in the coarse moduli space of principally polarized abelian varieties with
multiplication by a maximal order in a quaternion algebra over a totally real
number field. These are abelian varieties of even dimension with quaternion and 
complex multiplication. We describe them explicitly via isogenies to
products of abelian varieties of smaller dimension together with estimates on
the degree.
\end{abstract}


\section{Introduction} \label{intro}
We consider the coarse moduli space $ {\mathcal M}_D $ of principally polarized abelian 
surfaces $A$ over ${\mathbb C}$ with multiplication by ${\mathcal O}_D$, a maximal order 
in an indefinite quaternion algebra $D$ over ${\mathbb Q}$ (see \S \ref{qmcm}). We say 
such an abelian surface $A$ has quaternion multiplication (QM) by $D$. In fact, the moduli 
space ${\mathcal M}_D$ defines an algebraic curve $C$ over a number field, a so-called 
Shimura curve. Shimura curves are a natural generalization of modular curves. They play 
similar roles in number theory. For example, in analogy to CM points on classical modular
curves, one can use CM points on $C$ to construct class fields (see
\cite{Shimura:ClassFields}). These points on ${\mathcal M}_D$ correspond to abelian surfaces with even more endomorphisms, namely abelian
surfaces with quaternion multiplication and complex multiplication
(QM+CM). In this paper, we give a description of abelian surfaces of type QM+CM
over $F$ which is not restricted to $F={\mathbb C}$, but also applies to the
case of an arbitrary field of definition.\par 
First, we give an account of our description of abelian surfaces with QM+CM.
More details along with a generalization to higher dimensions can be found in
Section \ref{qmcm}. It follows from the classification of possible endomorphism
algebras of abelian varieties that an abelian surface with QM+CM is isogenous to
a product $\tilde A^2$, where $\tilde A$ is an elliptic curve with complex
multiplication. We construct an explicit isogeny
$\psi\colon  A\rightarrow \tilde
A_e \times \tilde A_{\overline e}$, where $\tilde A_e$ and  $\tilde A_{\overline
e}$
are elliptic curves isogenous to $\tilde A$. The construction goes as follows.
The center of ${\rm End}(A)$ is isomorphic to an order ${\mathcal O}_{L,c}$ in
an imaginary quadratic field $L$. We choose an embedding $\iota\colon {\mathcal
O}_{L,c}\hookrightarrow {\mathcal O}_D$ which corresponds canonically to an
idempotent $e\in D\otimes_{\mathbb Q} L$ (see \S \ref{splitting}). Essentially,
the idempotent defines the isogeny. In Section \ref{qmcm.2}, we show that the
elliptic curves $\tilde A_e$ and $\tilde A_{\overline e}$ have isomorphic
endomorphism rings. This is only shown in the case of an abelian surface $A$.
The proof goes as follows. We show that there exist two isogenies $u_i\colon
\tilde A_e \rightarrow \tilde A_{\overline e}$ of relatively prime degree. Then,
we use a result of Kohel (see \cite{Kohel:Diss}) to show that $\tilde A_e$ and
$\tilde A_{\overline e}$ have isomorphic endomorphism rings.\par
The idea for the construction of $\psi\colon A\rightarrow \tilde A_e \times \tilde A_{\overline e}$ goes back to
\cite{Mori:Expansion}, where this isogeny is constructed in the case $F=\mathbb{C}$. We generalize this to arbitrary field 
of definition $F$ and arbitrary even dimension $g$ of $A$.\par
In Section \ref{qmcm.2} we furthermore give another description of abelian
surfaces $A$ of type QM+CM using \cite{Kani:CMProd}. It is not generalizable to
higher dimensions. The description is closely linked to the construction of the
isogeny in Section \ref{qmcm}. Morally, $A$ is uniquely determined by a
suitable choice of integer $c$ and isomorphism class of a CM-elliptic curve $E'$.
This follows from the fact (\cite{Kani:CMProd}) that $A$ is isomorphic to a
product $E\times E'$ for suitable CM elliptic curves $E, E'$. For the right
choice of optimal embedding $\iota\colon {\mathcal O}_{L,c}\hookrightarrow
{\mathcal O}_D$ the elliptic curve $E'$ is isomorphic to the elliptic curve
$\tilde A_{\overline e}$ as above.\par
In \cite{BG}, Bayer and Gu{\`a}rdia use a different approach to construct fake
elliptic curves in the case $F=\mathbb{C}$. This leads to a different
description of fake elliptic curves, namely as Jacobians of hyperelliptic
curves. They give explicit equations using $\theta$-functions for those curves,
if the  abelian varieties correspond to certain CM points. This method only
works for dimension $g=2$ as in this case every principally polarized abelian
variety is in fact a Jacobian.\par
One application for the structure theorem for fake elliptic curves (Theorem \ref{thm:Kani}) to deformation theory is given in \cite{Diss}. There, 
we are interested in the $p$-adic geometry of CM points on the Shimura curve $C$, describing principally polarized
abelian surfaces of type QM. Details can also be found in Section \ref{qmcm.2}.
\section{Splitting of Quaternion Algebras and the Corresponding Idempotent}\label{splitting}
We start by giving the necessary background on quaternion algebras, especially on splitting fields. We then define the
idempotent corresponding to this splitting. This idempotent will be used in Section \ref{qmcm} to define the isogeny 
$\psi\colon  A\rightarrow \tilde A_e \times \tilde A_{\overline e}$.\par
First we fix some notation. Let $D$ be a quaternion algebra over a totally real field $K$, that is a central simple
algebra over $K$ of rank $4$ containing $K$. Let $h\mapsto \overline h$ denote the the standard involution on $D$. We denote by
$\Sigma$ the set of places of $K$ where $D$ is ramified. In other words $v\in\Sigma$ if and only if the quaternion
algebra $D\otimes_{K} K_v$ over the localization $K_v$ of $K$ at $v$ is a division algebra. We assume that $D$ is totally indefinite, that is $\Sigma$
contains no infinite place. Then $D$ is a possible endomorphism algebra of an abelian variety (see \cite[21
Theorem 2]{Mumford:AV}).\par
We study abelian varieties $A$ with even more endomorphisms, namely those which also have complex multiplication. Due to the aforementioned structure theorem for
endomorphism algebras, there exists an embedding $\iota\colon L\hookrightarrow D$ of a totally imaginary field $L$ with $[L:K]=2$ such 
that $D\otimes_{\mathbb Q} L\subset{\rm End}^0(A)$ (see \ref{pro:QMCM}).  By \cite[Theorem I.2.8]{Vigneras:Quaternions} this field $L$ is a \emph{splitting field} of $D$,
that is a field $L/K$ with $[L:K]=2$ such that $D\otimes_{\mathbb Q} L\simeq {\rm M}_2(L)$ holds or, equivalently (\cite[Theorem III.3.8]{Vigneras:Quaternions}), a quadratic
extensions $L$ of $K$ in which no prime ${\mathfrak p}$ corresponding to a place in $\Sigma$ is totally split.\par
Hence we are interested in the totally imaginary splitting fields $L$ of $D$. We fix the field $L$ and an embedding $\iota\colon L\hookrightarrow D$. 
The restriction to $L$ of the standard involution in $D$ is complex conjugation. We are interested in an explicit description of $D$ as $L$-algebra. By \cite[Chap.
I]{Vigneras:Quaternions} there exist $\theta\in K^\ast$ and $u\in D$ with the following properties:
\begin{align}
\begin{split}\label{def:u}
 	D&=\iota(L)\oplus u \iota(L),\\
	u^2&=\theta,\\
	u\iota(m)&=\iota(\overline m) u\quad  \forall m\in L.
\end{split}
\end{align}
This implies that $\overline u=-u$ holds.

\begin{remark}\label{def:Lu}
The $L$-algebra structure given by Eq.~(\ref{def:u}) determines $D$ up to
isomorphism (loc.~cit.). We denote this situation by $D=\left(\frac{L,\theta}{K}\right)$, resp. $D=\left(\frac{\iota(L),u^2}{K}\right)$ if we want to 
be explicit about which embedding $\iota\colon L\hookrightarrow D$ and which element $u$ as in Eq.~(\ref{def:u}) we are considering.\par
Of course, such $u\in D$ or even $\theta\in K^\ast$ are not unique.
For example, we could multiply $\theta$ by ${\rm n_{L/K}}(m)$ for $m\in L$, where ${\rm n}_{L/K}$
is the reduced norm of $L/K$. 
\end{remark}
In Section \ref{qmcm} we construct an isogeny between an abelian
variety and a product of abelian varieties of smaller dimension. In terms of
the endomorphism algebra this corresponds to the determination of idempotents.
Therefore, we are interested in non-trivial idempotents $e\in D\otimes_K
L\simeq {\rm M}_2(L)$. We now describe how to explicitly construct such
idempotents.\par
There exists an isomorphism $\kappa$ between $M:=D\otimes_K L$ and the
$L$-linear maps ${\rm Hom}_L(D)$ (\cite[III, 5.1.13]{Knus:QuadHermitForms}).
Every non-trivial idempotent $e\in M$ corresponds via $\kappa$ to a
projection $p_e\colon  D\rightarrow V$, where $V$ denotes an one-dimensional $L$-subspace of
$D$. The projection $p_e$ is orthogonal with respect to the inner product
defined by the reduced trace on $D$, as ${\rm tr}_{M/L}(e)=1$ holds. 

\begin{lemma}\label{idem}
Let $\alpha\in L$ be an arbitrary element satisfying ${\rm tr}_{L/K}(\alpha)=0$ and denote $\delta:=\alpha^2\in K$. Then there exists a bijection of sets
\begin{align*}
\xy\xymatrix@M=0pt@R=10pt@W=10pt{%
*{\{\iota\colon \;L\hookrightarrow D\}} &
\ar[r]
 & &
*{\{\begin{aligned}\text{non-trivial Idempotents } e\in M\end{aligned}\}}\\
*{\iota\colon \;L\hookrightarrow D} & \ar@{|->}[r]  & &
*{e_\iota=\tfrac{1}{2}(1\otimes
1+\iota(\alpha)^{-1}\otimes\alpha),}\\
*{\left\{\begin{aligned}
\iota\colon \;L&\hookrightarrow D\\
\alpha&\mapsto ab^{-1}
\end{aligned}\right.}& 
  & \ar@{|->}[l] & *{e=a\otimes 1+b\otimes\alpha.}%
}
\endxy
\end{align*}
Moreover, the bijection is independent of the choice of $\alpha$.
\end{lemma}
\begin{proof}
To ease notations, we denote by ${\rm n}$ and ${\rm tr}$ both the reduced norm resp.\ trace of $M$ over $L$ and the reduced norm
resp.\ trace of $D$ over $K$. 
 First let $e:=a\otimes 1+b\otimes \alpha\in M$ denote a non-trivial idempotent. We calculate
\begin{align*}
 e=e^2=(a\otimes 1+b\otimes \alpha)^2=(a^2+b^2\delta)\otimes 1+(ab+ba)\otimes\alpha,
\end{align*}
and hence the identities
\begin{align}
a^2+b^2\delta&=a\label{idem.gl.1}\\
ab+ba&=b && \Leftrightarrow   & ab^{-1}&=b^{-1}(1-a)\label{idem.gl.3} 
\end{align}
hold. Using Eqs.~(\ref{idem.gl.1}) and (\ref{idem.gl.3}) we calculate
\begin{align*}
  (ab^{-1})\cdot(ab^{-1})&\stackrel{(\ref{idem.gl.3})}{=}(b^{-1}(1-a))\cdot(ab^{-1})\\
	&=b^{-1}ab^{-1}-b^{-1}a^2b^{-1}\\
	&\stackrel{(\ref{idem.gl.1})}{=}
        b^{-1}ab^{-1}-b^{-1}(a-b^2\delta)b^{-1}\\
	&=b^{-1}ab^{-1}-b^{-1}ab^{-1}+\delta\\
	&=\delta.
\end{align*}
We calculate
\begin{align*}
0={\rm n}(e)=e\overline e=({\rm n}(a)+{\rm n}(b)\delta)\otimes 1+{\rm tr}(a\overline b)\otimes
 \alpha,
\end{align*}
and hence ${\rm tr}(ab^{-1})={\rm tr}(a\overline b)/{\rm n}(b)=0$. We conclude that
$\alpha\mapsto ab^{-1}$ defines an embedding.\\
Given the element $e_\iota:=\tfrac{1}{2}(1\otimes 1+\iota(\alpha)^{-1}\otimes\alpha)\in M$, we calculate
\begin{align*}
  e_\iota^2=\frac{1}{4}\left((1+\iota(\delta)^{-1}\cdot\delta)\otimes
1+2\cdot\iota(\alpha)^{-1}\otimes\alpha\right)=e_\iota.
\end{align*}
Thus $e_\iota$ is a idempotent which is non-trivial as ${\rm tr}(e_\iota)=1$. The two maps are obviously
arrow-reversing. As the map from the left to the right is injective we conclude that it is in fact a bijection.
It is obvious that the maps of the correspondence are independent of the choice of $\alpha\in L$ with ${\rm tr}(\alpha)=0$.
\end{proof}

\section{Abelian Varieties with Quaternion and Complex Multiplication}\label{qmcm}

Let $g\in{\mathbb N}$ be even, $K$ a totally real field of degree $g/2$ over ${\mathbb Q}$ and let $R$ denote its ring
of integers. Let $D$ denote an indefinite quaternion algebra over $K$ and $\Sigma=\{{\mathfrak p}_1,\dots,{\mathfrak
p}_r\}\ne\emptyset$ the set of ramified primes. We denote by $\dagger\colon D\rightarrow D$ a positive involution on $D$
and by ${\mathcal O}_D$ a maximal $R$-order of $D$.

\begin{definition}
Let $A$ be a principally polarized abelian variety of dimension $g$ over a field $F$. We call $A$ of \emph{type CM} if there exists an embedding 
$\iota\colon L\rightarrow {\rm End}^0(A):={\rm End}(A)$ of a CM field $L$ of dimension $2g$.\\
We call $A$ of \emph{type QM} by $D$ if there exists an embedding $\psi_{\rm QM}\colon {\mathcal O}_D\hookrightarrow{\rm End}(A)$ (of rings) such that the
involution $\dagger$ on $D$ corresponds to the Rosati involution on ${\rm End}^0(A)\otimes{\mathbb Q}$.
The embedding $\psi_{\rm QM}$ is called the \emph{QM-type} of $A$.
\end{definition}

In the rest of the section we fix the following notation. Let $F$ denote an algebraically closed field (of arbitrary characteristic). Let $A$ be a principally
polarized abelian variety over $F$ of even dimension $g$ with QM-type $\psi_{\rm QM}\colon  {\mathcal
O}_D\hookrightarrow {\rm End}(A)$. If $F$ is a field of positive characteristic we assume $A$ to be ordinary. Furthermore, let $A$ 
be of CM type (by a CM field) $\tilde L$. We want to show that in this case we have an embedding $D\otimes_{\mathbb Q} L\hookrightarrow {\rm End}^0(A)$
where $L$ denotes an imaginary splitting field of $D$.

\begin{proposition}\label{pro:QMCM}
 Let $A$ be as above. Then the following holds.
\begin{enumerate}
 \item[(i)] $A$ is isogenous to the product $B^n$ of a nontrivial simple subvariety $B$ of dimension $g/n$ with ${\rm End}^0(B)\simeq L'$, where $L'$ is a CM field of dimension $2g/n$. Furthermore  
      $n$ is even.
 \item[(ii)] There exists an embedding $\varepsilon\colon  D\otimes_K L\hookrightarrow{\rm End}^0(A)$ of the totally imaginary splitting field $L:=KL'$ of $D$ 
      such that $\varepsilon(1\otimes_K L)$ contains the center of ${\rm End}^0(A)$.
\end{enumerate}
\end{proposition}
\begin{proof} 
      As $A$ has complex multiplication by a field we conclude by \cite[\S 5 Prop.\ 3]{Shimura:CM} that $A$ is isogenous to a product 
      $B^n$ of a simple subvariety $B$. If we denote by $Z$ the center of ${\rm End}^0(B)$, then \cite[\S 5 Prop.\ 4]{Shimura:CM} states that
      \begin{align} 
	[{\rm End}^0(B):Z]\cdot [Z:\mathbb Q] = 2g/n.\label{eq:cm} 
      \end{align}
      By the classification of endomorphism algebras of simple abelian varieties (see \cite[p.\ 202]{Mumford:AV}) 
      it follows that ${\rm End}^0(B)$ is either:
      \begin{enumerate}
       \item[(a)] isomorphic to a CM field $L'$ of dimension $2g/n$, usually called Type IV$(g/n,1)$,
       \item[(b)] isomorphic to a quaternion algebra $\tilde D$ over a CM field $\tilde L$ of dimension $g/n$, Type IV$(g/2n,2)$.
       \item[(c)] isomorphic to a definite quaternion algebra over a (totally) real number field $K'$ with $[K':\mathbb Q]=g/n$ , Type III$(g/n)$,
      \end{enumerate}
      First we show that in any characteristic only case (c) is possible. In characteristic $0$ this follows
      from the fact that $[{\rm End}^0(B):\mathbb Q]\mid 2g/n$. In characteristic $p$ we use the following reasoning. The simple abelian variety $B$ 
      has sufficiently many endomorphism in the sense that Eq.~(\ref{eq:cm}) is satisfied. By \cite[Theorem 1.1]{Oort:CMAV}, we may then assume
      that $B$ is defined over a finite field. As $A$ and hence $B$ are ordinary we conclude by \cite[Prop.\ 3.14]{Oort:Endomorphism-Algebras} that 
      ${\rm End}^0(B)$ is commutative.\\
      Hence ${\rm End}^0(B)$ is a CM field $L'$ of dimension $2g/n$. Denote by $K'$ its maximal totally real subfield, whence 
      $L'/K'$ is an imaginary quadratic extension. The field $L'$ is also the center 
      of ${\rm End}^0(A)\simeq {\rm M}_n({\rm End}^0(B))$. Therefore the subalgebra $D':=D\otimes_{K} KL'$ of ${\rm End}^0(A)$ is a quaternion algebra 
      over the field $KL'\subset {\rm End}^0(A)$. By \cite[Prop.\ 5.4]{Shimura:CM} again $[KL':K]\leq 2$ and 
      as $L'$ is an imaginary field we conclude that $KK'=K$ and $KL'/K$ is an imaginary quadratic extension. 
      As $[K:\mathbb Q]=g/2$ and $[K':\mathbb Q]=g/n$ it follows that $[K:K']=n/2$. In particular, $n$ is even and greater than $1$.
      It remains to show that $D'={\rm M}_2(KL')$ or, equivalently, that there exists an embedding $\tilde \iota\colon L'\rightarrow D$.
      We can apply \cite[Theorem 4.11]{Jacobson:Algebra2} as $D'$ is contained in ${\rm M}_n(L')$ as $L'$-subalgebra. Let $C$ denote the centralizer of
      $D'$ in ${\rm M}_n(L')$. It contains $KL'$. By \cite[Theorem 4.11]{Jacobson:Algebra2} we conclude that $C=KL'$ and, by applying the double centralizer 
      Theorem (\cite[Theorem 4.10]{Jacobson:Algebra2}), that $D'\simeq {\rm M}_2(KL')$.
\end{proof}

\begin{figure}[ph]
 \xymatrix{
    & DL'\ar@{-}[d]^{2^2} &\\
    D \ar@{^{(}->}[ur] \ar@{-}[d]_{2^2} & KL' \ar@{-}[dl]_{2} \ar@{-}[d]_{n} \ar@{=}[r] &C \ar@{-}[dl]^{n}\\
    K \ar@{-}[dr]_{n/2}  &  L'\ar@{-}[d]^{2} &\\
    &  K'\ar@{-}[d]^{g/n} &\\
    &  \mathbb Q &%
 }
\vspace*{8pt}
\caption{Diagram indicating the algebraic structure of 
      the subrings of ${\rm End}^0(A)$ mentioned above.\label{fig1}}
\end{figure}
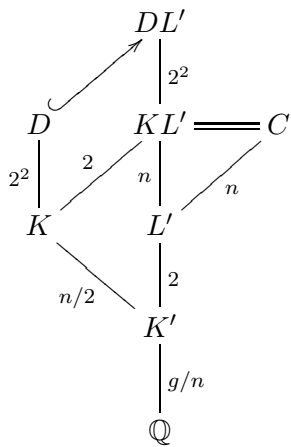

\begin{remark}
\begin{enumerate}
 \item
  In the case of characteristic $0$ Proposition \ref{pro:QMCM} can be found in \cite[\S 9.9]{Shimura:ClassFields}. There
  it is phrased in terms of special points on Shimura varieties of type QM and the corresponding abelian varieties. Another
  reference for the statement of Proposition \ref{pro:QMCM} without proof is \cite[\S 5]{Rotger:ForgetfulMaps}.
  \item
  Note that in Proposition \ref{pro:QMCM}. $KL'$ is the only possible choice for a splitting field $L$ of $D$ such that $D\otimes_{\mathbb Q} L$
  can be embedded in ${\rm End}^0(A)$. Furthermore, the embedding $\varepsilon\colon  D\otimes_K L\hookrightarrow{\rm End}^0(A)$ is uniquely determined (up to conjugation in $1\otimes L'$) by
  the QM type $\psi_{\rm QM}\colon D\hookrightarrow {\rm End}^0(A)$.
\end{enumerate}
\end{remark}

In the following we are interested in constructing an explicit isogeny $A\rightarrow \tilde A^2$ for the QM+CM abelian variety $A$, where $\tilde A$ is an abelian variety of dimension $g/2$.
Let in the following $\varepsilon\colon  D\otimes_K L\hookrightarrow{\rm End}^0(A)$ be fixed. We denote by $R\otimes_R {\mathcal O}$ the pre-image of the 
center of ${\rm End}(A)\cap\psi_{\rm QM}(D)$ under $\varepsilon$. Then ${\mathcal O}\subset L$ is some $R$-order of $L$. 
We denote ${\mathfrak c}\subset{\mathcal O}_L$ its conductor in the maximal order ${\mathcal O}_L\supset{\mathcal O}$ of $L$. 
In the following we write ${\mathcal O}_{L,{\mathfrak c}}$ for ${\mathcal O}$. Then 
\begin{align*}
 \varepsilon\colon {\mathcal O}_D\otimes_R {\mathcal O}_{L,{\mathfrak c}}\hookrightarrow {\rm End}(A)
\end{align*}
is an embedding.\par
We call an embedding $\iota\colon L\hookrightarrow D$ \emph{optimal embedding} of ${\mathcal O}_{L,{\mathfrak c}}\subset
L$ if $\iota^{-1}({\mathcal O}_D)={\mathcal O}_{L,{\mathfrak c}}$. By \cite[Cor.\ III.5.12]{Vigneras:Quaternions}, there exists an optimal embedding
\begin{align*}
\iota\colon L\hookrightarrow D 
\end{align*}
of ${\mathcal O}_{L,{\mathfrak c}}$ if and only if the conductor $\mathfrak c$ is not contained in any prime ideal $\mathfrak p$ of $L$ such that $\mathfrak p\cap K$ 
is in the ramification locus $\Sigma$ of $D/K$. By our assumptions there exists an optimal embedding of ${\mathcal O}_{L,{\mathfrak c}}$ and in the following we fix one.
Let $\alpha\in {\mathcal O}_{L,{\mathfrak c}}$ be such that ${\rm tr}_{L/K}(\alpha)=0$ is satisfied, and denote
\begin{align*}
 e:=\frac{1}{2}(1\otimes1+\iota(\alpha)^{-1}\otimes\alpha)\in D\otimes_K L
\end{align*}
the unique idempotent corresponding to the embedding $\iota\colon L\hookrightarrow D$ as in Proposition \ref{idem}. We further
denote by $\overline e\in D\otimes_K L$ the conjugate of $e$ w.r.t.\ the standard involution on $D\otimes_K L$. By the definition of $e$, the identity $\overline e=1-e$
holds. In other words $\overline e$ is the complement of the idempotent $e$ as element of ${\rm Hom}_{L}(D)$. Note that the elements $2e \alpha$ and
$2\overline e \alpha$ are in $\iota({\mathcal O}_{L,{\mathfrak c}})\otimes_R {\mathcal O}_{L,{\mathfrak c}}$.\par
We introduce the following notation. If $I\subset{\rm End}(A)$ is an ideal we denote by $A[I]$ the subgroup scheme
defined as the intersection of the kernels of all elements in $I$. Note that if we write ideal we always mean left ideal. 
It is easy to see that $A[I]$ is finite if and only if $I$ contains an isogeny.
\begin{theorem}\label{haupt}
Let the notation and assumptions be as above. In particular $A$ denotes an abelian variety over the field $F$ of even dimension $g$ with QM
and CM, which is ordinary if ${\rm char}(F)>0$. Then there exist isogenous abelian varieties $\tilde A_e,\;\tilde A_{\overline e}$ of dimension $g/2$ 
with complex multiplication by (at least) $\mathcal O_{L,\mathfrak c}$ and for every element $\alpha\in {\mathcal O}_{L,\mathfrak c}$ with ${\rm tr}_{L/K}=0$ 
an isogeny 
\begin{align*}
\psi_{e\alpha}\colon A\rightarrow \tilde A_e\times \tilde A_{\overline e}.
\end{align*}
with ${\rm deg}(\psi_{e\alpha})\mid 4^g\cdot{\rm n}_{L/{\mathbb Q}}(\alpha)^2$.
\end{theorem}
\begin{proof} 
First we construct independent abelian subvarieties of $A$ of dimension $g/2$ with CM by ${\mathcal O}_{L,{\mathfrak c}}$.\\
For this purpose let $\eta\in\iota({\mathcal O}_{L,{\mathfrak c}})$ with ${\rm tr}(\eta)=0$ be such that
$e\eta\in{\rm End}(A)$ holds. As $e$ is a non-trivial idempotent, the image $\tilde A_{e\eta}:=e\eta(A)$ is a
non-trivial abelian subvariety of $A$. Analogous $\tilde A_{\overline{e\eta}}:=\overline{e\eta}(A)$ is a non-trivial
abelian subvariety. If we take a different element $\eta'\in\iota({\mathcal O}_{L,{\mathfrak c}})$ with ${\rm tr}(\eta')=0$ 
then $\tilde A_{e\eta}$ (resp.\ $\tilde A_{\overline{e\eta}}$) and $\tilde A_{e\eta'}$ (resp.\ $\tilde A_{\overline{e\eta'}}$) 
coincide as subvarieties of $A$. Hence w.l.o.g.\ we can choose $\eta:=2\alpha$ where $\alpha\in{\mathcal O}_{L,\mathfrak c}$
is an element of trace ${\rm tr}_{L/K}(\alpha)=0$ as in the statement of the theorem. Hence we simply write $\tilde{A_e}$ 
(resp.\ $\tilde{A_{\overline e}}$) for $\tilde A_{e\eta}$ (resp.\ $\tilde A_{\overline{e\eta}}$).\\
We claim that $\tilde A_{e}$ and  $\tilde A_{\overline{e}}$ are isogenous.
Let $u\in D$ be an element such that Eq.~(\ref{def:u}) holds. Without loss of generality, we may suppose that $u\in{\mathcal
O}_D$, after multiplication by an appropriate $N\in{\mathbb N}$. Then $\overline{e\eta}u=u e\eta$ holds in ${\rm
End}(A)$, that is the diagram
\begin{align*}
\xymatrix{ {A}\ar[d]^{u} \ar^{e\eta}[r]
& {\tilde A_{e}} \ar[d]^{u|_{\tilde A_{e}}} \\
{A} \ar[r]^{\overline e\overline\eta} & {\tilde A_{\overline{e}}}.
}\label{iso.ell.cur}
\end{align*}
is commutative. It is shown in \cite[\S 7 Prop.\ 7]{Shimura:CM} that the morphism $u$ has degree ${\rm deg}(u)={\rm
n}_{D/{\mathbb Q}}(u)={\rm n}_{K/{\mathbb Q}}(\theta)$, where $\theta=u^2$ is some element in $K$. 
We conclude that $\tilde u:=u|_{\tilde A_{e}}$ is an isogeny and that its degree ${\rm deg}(\tilde u)$
divides ${\rm n}_{K/{\mathbb Q}}(\theta)$. As $1=e+\overline e$ it follows that $A_{e}$ and $\tilde
A_{\overline{e}}$ both have dimension $g/2$.\\
By an easy computation one checks that the commutator of $e\eta$ and the commutator of $\overline{e\eta}$ in ${\rm End}(A)$ contain $\iota({\mathcal
O}_{L,{\mathfrak c}})\otimes_R {\mathcal O}_{L,{\mathfrak c}}$. We conclude that $\tilde A_{e}$ and $\tilde A_{\overline{e}}$ have 
complex multiplication by at least ${\mathcal O}_{L,{\mathfrak c}}\hookrightarrow {\rm End}(\tilde A_{e})$. 
Next we construct an isogeny $\psi_{e\eta}\colon \; A\longrightarrow  \tilde A_{e}\times \tilde A_{\overline{e}}$ as
in the statement of the theorem.\\
Let $I:=\langle e\eta,\overline{e\eta}\rangle$ be the left ideal in ${\rm End}(A)$ generated by $e\eta$ and
$\overline{e\eta}$. The subgroup scheme $A[I]\subset A$ is finite as the element $\eta=e\eta-\overline{e\eta}\in I\cap
{\mathcal O}_D$ is an isogeny. Hence $A\rightarrow A/A[I]$ is an isogeny, which by \cite[\S 7 Prop.\ 7]{Shimura:CM} is
given by 
\begin{align*}
 \psi_{e\eta}\colon  &A\longrightarrow  \tilde A_{e}\times \tilde
A_{\overline{e}} \subset A^2\\
	&P\longmapsto  (e\eta(P),\overline{e\eta}(P)).
\end{align*}
We denote by ${\rm n}_{L/{\mathbb Q}}\colon L\rightarrow {\mathbb Q}$ the reduced norm of $L$ over ${\mathbb Q}$. Then it
is shown in \cite[\S 7 Prop.\ 10]{Shimura:CM} that ${\rm deg}(\psi_{e\eta})\mid {\rm deg}(\eta)={\rm n}_{L/{\mathbb
Q}}(\eta)^2$ holds.
\end{proof}
\begin{remark}
\begin{enumerate}
 \item %
    The fact that $A$ of dimension $g$ is isogenous to two isogenous abelian subvarieties of dimension $g/2$
    with CM already follows from Proposition \ref{pro:QMCM}. The interesting part of Theorem \ref{haupt} is the special choice
    of idempotent $e$. So, an isogeny $\psi_e$ together with an estimate on its degree can be given without much 
    knowledge of the algebraic structure of ${\rm End}(A)$.
 \item %
    In the definition of an abelian variety $A$ of QM-type we assumed that a
    maximal order $\mathcal O_D\subset D$ is contained in ${\rm End}(A)$. But,
    by the proof above, we see that for the construction of the isogeny
    $\psi_{e\eta}\colon A\rightarrow \tilde A_e\times \tilde A_{\overline e}$
    it suffice that there exists an isomorphism $D\otimes_{\mathbb Q} L
    \xrightarrow{\sim} {\rm End}^0(A)$. In the extended case it is more
    difficult to give estimates on the degree of $\psi_{e\eta}$.
\end{enumerate}
\end{remark}

\section{Case of Fake Elliptic Curves}\label{qmcm.2}

In this section we first apply Theorem \ref{haupt} to the case of so called fake elliptic curves, that is abelian surfaces
of type QM+CM. We then give an explicit construction of these abelian surfaces. We first need the following lemma whose proof
is a corollary of the theorem of arithmetic progression. Recall that we assume the abelian surface $A$ to be ordinary if it is
defined over a field $F$ of positive characteristic. As was already mentioned in the proof of Proposition \ref{pro:QMCM} we
may in this case assume that $F$ is a finite field.\\
In the following we want to find conditions on $\theta\in\mathbb N$ such that $D=\left(\frac{\Delta_L,\theta}{\mathbb Q}\right)$.
Therefor, denote for a prime $p$ by $v_p(x):=\max\{n\colon p^n\mid x\}$ the valuation of $x$ at $p$.

\begin{lemma}\label{lemma.pre.arith}
 Let $\theta$ be a natural number such that $D\simeq\left(\frac{\Delta_L,\theta}{\mathbb Q}\right)$.
 Denote $\Delta_L:={\rm disc}(L)$ and $m_0$ the product of odd primes $p\mid{\rm disc}(D)$ with $p\nmid \Delta_L$. 
 Then $m_0\mid \theta$ and $m:=\frac{\theta}{m_0}$ satisfies the following properties:
  \begin{enumerate}
    \item $m$ is positive.
    \item If $p\nmid m_0 \Delta_L$ is an odd prime then $p\nmid {\rm disc}(D)$.
    \item If $p\mid m_0$ is an odd prime then $v_p(m)$ is even.
    \item If $p\mid \Delta_L$ is an odd prime then 
      \[ (\Delta_L,\theta)_p= (-1)^{v_p(m)\varepsilon(p)} \left(\frac{\Delta_L/p}{p}\right)^{v_p(m)}\cdot \left(\frac{m_0}{p}\right)\left(\frac{m/p^{v_p(m)}}{p}\right)\]
      holds.
    \item For $p=2$ we make the following distinction:
    \begin{enumerate}
      \item In the case $\Delta_L\equiv 5 \pmod 8$ it holds that $2\mid{\rm disc}(D)$ if and only if $v_2(m)$ is odd.
      \item In the case $\Delta_L\equiv 1 \pmod 8$ it holds that $2\nmid{\rm disc}(D)$ and there are no further restrictions on $m$.
      \item In the case $-d\equiv 3 \pmod 4$ it holds that $2\mid{\rm disc}(D)$ if and only if $\varepsilon(\tilde\theta)+v_2(m)\omega(-d)\equiv 1 \pmod 2 $ .
      \item In the case $-d$ even it holds that $2\mid {\rm disc}(D)$ if and only if $\varepsilon(-d/2)\varepsilon(\tilde\theta)+\omega(\theta)+v_2(m)\omega(-d/2)\pmod 2$.
    \end{enumerate}
  \end{enumerate}
  On the other hand, if we chose an $m\in\mathbb N$ satisfying 1.-5. then $D\simeq\left(\frac{\Delta_L,\theta}{\mathbb Q}\right)$ holds.
\end{lemma}
\begin{proof}
 The proof is a simple calculation using the Hilbert symbol using the following well-known formulas. Decompose two integers $a,b$ as 
 $a=p^{v_p(a)}\tilde a$,  $b=p^{v_p(b)}\tilde b$ and for $x\in\mathbb Z$ denote
  \begin{align*}
    \varepsilon(x)=(x-1)/2, && \omega(x)=(x^2-1)/8.
  \end{align*}
  Then we can calculate the Hilbert symbol for a prime $p$ as
  \begin{align}
    (a,b)_p&=(-1)^{v_p(a)v_p(b)\varepsilon(p)} \left(\frac{\tilde a}{p}\right)^{v_p(b)} \left(\frac{\tilde b}{p}\right)^{v_p(a)},
    && \text{if $p$ is an odd prime},\\
    (a,b)_2&=(-1)^{\varepsilon(\tilde a)\varepsilon(\tilde b)+v_2(a)\omega(\tilde b)+v_2(b)\omega(\tilde a)},&& \text{for $p=2$}.
  \end{align}
  We go through the different primes $p$.
 \begin{enumerate}
  \item This follows as $D$ is indefinite.
  \item If $p\nmid 2m_0 \Delta_L$ we calculate
    \[ (\Delta_L,\theta)_p= \left(\frac{\Delta_L}{p}\right)^{v_p(m)}.\]
    Assume that $p\mid{\rm disc}(D)$ then $\left(\frac{\Delta_L}{p}\right)$ must be odd. But this contradicts the assumption that $p\nmid m_0$.
  \item If $p\mid m_0$ then $p\nmid \Delta_L$ and we calculate
    \[ (\Delta_L,\theta)_p= \left(\frac{\Delta_L}{p}\right)^{v_p(m)+1}.\]
    As $p\mid m_0$ we know that $\left(\frac{\Delta_L}{p}\right)=-1$ and $(\Delta_L,\theta)_p=-1$. We conclude that $v_p(m)$ must be even.
  \item If $p\mid \Delta_L$ odd then $p\nmid m_0$ and we easily calculate the formula as in the statement.
  \item If $p=2$ we have to make a case by case study.
    \begin{enumerate}
     \item If $\Delta_L\equiv 5 \pmod 8$ then 
      \[v_p(\Delta_L)=0,\qquad \varepsilon(\Delta_L)\equiv 0\pmod 2,\qquad \omega(\Delta_L)\equiv 1\pmod 2,\]
	and we have the formula
      \[ (\Delta_L,\theta)_2=(-1)^{v_2(m)}.\]
     \item If $\Delta_L\equiv 1 \pmod 8$ then the prime $p=2$ splits in $L$, whence $2\nmid {\rm disc}(D)$.
     \item If $-d\equiv 3 \pmod 4$ then
      \[\varepsilon(-d)\equiv 1 \pmod 2,\qquad \omega(-d)\equiv_2\left\{ \begin{matrix} 1, & \text{if $-d\equiv 3 \pmod 8$},\\
	    0, & \text{if $-d\equiv 7 \pmod 8$}.\end{matrix}\right.\] 
      We conclude that
      \[ (\Delta_L,\theta)_2=(-1)^{\varepsilon(\tilde\theta)+v_2(m)\omega(-d)}.\]
     \item If $-d$ is even it follows that
      \[ (\Delta_L,\theta)_2=(-1)^{\varepsilon(-d/2)\varepsilon(\tilde\theta)+\omega(\theta)+v_2(m)\omega(-d/2)}.\]
    \end{enumerate}
 \end{enumerate}
\end{proof}

\begin{lemma}\label{lemma.arith}
  Denote $\Delta_L:={\rm disc}(L)$ and $m_0$ the product of odd primes $p\mid{\rm disc}(D)$ with $p\nmid \Delta_L$. 
 Then there exist two numbers $m_1,m_2\in\mathbb N$ without common divisor, coprime to $2m_0\Delta_L$, such that 
  \begin{itemize}
   \item $D\simeq\left(\frac{\Delta_L,m_0m_i}{\mathbb Q}\right)$ for $i=1,2$ if $\Delta_L\not\equiv 5 \pmod 8$,
   \item $D\simeq\left(\frac{\Delta_L,2^s m_0m_i}{\mathbb Q}\right)$ for $i=1,2$ if $\Delta_L\equiv 5 \pmod 8$, where $s=0$ if $2\nmid{\rm disc}(D)$ 
      and $s=1$ if $2\mid{\rm disc}(D)$.
  \end{itemize}
\end{lemma}
\begin{proof}
   We may choose $m_i$ coprime to $m_0\Delta_L$. In this case we must satisfy a finite number of equations modulo $2m_0\Delta_L$, namely
    \begin{align*}
      m_i&\not\equiv 0 \pmod {m_0},\\
      (\Delta_L,\theta)_p&= \left(\frac{m_0}{p}\right)\left(\frac{m_i}{p}\right), && \text{for $p\mid\Delta_L$},
    \end{align*}
    and a congruence relation modulo $2$, $4$ or $16$ (depending on the congruence class of $\Delta_L$ modulo $16$) which is determined by 5. in Lemma \ref{lemma.pre.arith}. This can be solved via the Chinese Remainder Theorem.
\end{proof}
Now we can apply Theorem \ref{haupt} to the case of fake elliptic curves, which gives further information about the endomorphism algebras involved.
\begin{corollary}[of Theorem \ref{haupt}]\label{case:ell}
Let the notation and assumptions be as in Theorem \ref{haupt}. Assume that the
dimension of $A$ is $2$. Let $d\in{\mathbb N}$ be the squarefree integer with
$L\simeq{\mathbb Q}(\sqrt{-d})$ and denote $m_0\in\mathbb N$ the product of odd primes
$p\mid{\rm disc}(D)$ with $p\nmid{\rm disc}(L)$. Then there exist isogenous
elliptic curves $E_e,\;E_{\overline e}$  with complex multiplication by ${\rm
End}(E_e)\simeq{\rm  End}(E_{\overline e})$ such that there is an isogeny 
\begin{align*}
\psi_e\colon A\rightarrow E_e\times E_{\overline e} 
\end{align*}
with ${\rm deg}(\psi_e)\mid (4c^2\cdot d)^2$. Assume furthermore that $D$ is
isomorphic to the quaternion algebra denoted $\left(\frac{L,m_0}{\mathbb
Q}\right)$ from Remark \ref{def:Lu}, that is $D$ is generated as ${\mathbb Q}$
module by $1,x,y,xy\in D$ such that $x^2=-d,\; y^2=m_0,\; xy=yx$. Then $E_e$
and $E_{\overline e}$ are even isomorphic.
\end{corollary}
\begin{proof}
Without loss of generality, the order ${\mathcal O}_{L,c}$ is given by ${\mathbb
Z}+c\cdot {\mathcal O}_{L}$, where ${\mathcal O}_{L}$ denotes the maximal order
in $L$. We choose $\alpha:=2c\cdot \sqrt{-d}\in{\mathcal O}_{L,c}$ in the
statement of Theorem \ref{haupt} and conclude that there exists an isogeny
$\psi_e\colon A\rightarrow E_e\times E_{\overline e}$ with ${\rm
deg}(\psi_e)\mid{\rm n}_{L/{\mathbb Q}}(\alpha)=4c^2\cdot d$, where
$E_e:=\tilde A_e$ (resp.\ $E_{\overline e}:=\tilde A_{\overline e}$) in the
notation of Theorem \ref{haupt} .\\
We want to show that
${\rm End}(E_e)\simeq{\rm End}(E_{\overline e}) \supset{\mathcal O}_{L,c}$. 
First we consider the case that $\Delta_L\not\equiv 5 \pmod 8$.
Let $m_i$ be two integers as in Lemma \ref{lemma.arith}. Hence there exist elements 
$u_i\in D$ for $i=1,2$ with $u_i^2=\theta_i:=m_im_0$ which satisfy Eq.~(\ref{def:u}) 
with respect to the embedding 
$\iota_e\colon {\mathcal O}_{L,c}\hookrightarrow {\mathcal O}_D$. 
Without loss of generality, we can assume that $u_i\in {\mathcal O}_D$.
Now, we consider the isogenies 
$\psi_{\rm QM}(u_1),\psi_{\rm QM}(u_2)\colon A\rightarrow A$ of the proof of
Theorem \ref{haupt} and their restrictions $\varphi_i:=\psi_{\rm QM}(u_1)|_{E_e}$ to $E_{e}$. 
As $u_i^2=\theta_i$ we conclude that ${\deg}(\varphi_i)$ divides $\theta_i^2=m^2m_i^2$. We want 
to decompose $\varphi_i$ into isogenies of degree $l^2$, where $l\mid m_0$ 
is a prime, or $\varphi_i$ has degree $m_i$.  
In concrete terms, let $m_0=\prod\limits_{j=1\dots n} l_j$ be a prime decomposition of $m_0$.
We factor $\varphi_i$ as composite 
\begin{align*}
  E_{e}\xrightarrow{\varphi_{i,0}} E_{i,1} \xrightarrow{\varphi_{i,1}} \dots \ \xrightarrow{\varphi_{i,n}} E_{\overline e},
\end{align*}
where $\varphi_{i,n}$ is an isogeny of degree ${\rm deg}(\varphi_{i,n})\mid m_i^2$, and the other isogenies
$\varphi_{i,j}$ for $j=0,\dots,n-1$ are isogenies of degree ${\rm deg}(\varphi_{i,j})\mid l_{j(i)}$.
First we consider the isogenies $\varphi_{i,j}$ for $j=0,\dots,n-1$. As was mentioned in the last section, 
$m_0$ and $c$ are coprime as ${\mathcal O}_{L,c}$ embeds optimally into $\mathcal O_D$. From Theorem 
\ref{haupt} follows that the endomorphism rings of $E_{e}$ and $E_{\overline e}$ are maximal at 
$l_{j(i)}$. We conclude by \cite[Prop.\ 21]{Kohel:Diss} that ${\rm End}(E_{e})={\rm End}(E_{i,j})$. 
We have chosen $m_0$ such that $l_j$ is inert. Then by 
\cite[Prop.\ 23]{Kohel:Diss} the isogenies $\varphi_{i,j}$ must be isomorphisms or 
${\rm deg}(\varphi_{i,j})=l_{j(i)}^2$, so in any case they are endomorphisms.
Next we consider the isogenies $E_{e}\xrightarrow{\varphi_{i,n}} E_{\overline e}$. By construction
the degree of $\varphi_{1,n}$ and the degree of $\varphi_{2,n}$ are coprime. By \cite[Prop.\
22]{Kohel:Diss} we conclude that the endomorphism rings of $E_{e}$ and $E_{\overline e}$ are 
isomorphic. Considering the last statement, if $\theta_i=m$ for one $i$, that is if 
$D=\left(\frac{L,m}{\mathbb Q}\right)$, then $\varphi_{i,n}\colon E_e\rightarrow E_{\overline e}$ is 
an isomorphism.\\
Now we treat the case $\Delta_L\equiv 5 \pmod 8$. If $2\nmid{\rm disc}(D)$ then nothing changes. So assume
that $2\mid{\rm disc}(D)$. In this case $2$ is inert in $L$, and as $2\nmid{\rm disc}(D)$ we know that 
$2$ and $c$ are coprime. We conclude as above that $\varphi_{i,n}$ must factor as composite of an endomorphism
of $E_{e}$ and an isogeny $\tilde \varphi_{i,n}\colon E_e\rightarrow E_{\overline e}$ of degree 
${\rm deg}(\tilde\varphi_{i,j})\mid m_i^2$. As the $m_i$ are coprime we conclude that 
the endomorphism rings of $E_{e}$ and $E_{\overline e}$ are isomorphic.
\end{proof}
\begin{remark}
\begin{enumerate}
 \item The statements \cite[Prop.\ 21]{Kohel:Diss} and \cite[Prop.\
  22]{Kohel:Diss} concern ordinary elliptic curves over finite fields. But
  the statement also hold in characteristic zero.
 \item Also note, that by the proof of the corollary in any case there exists a
    integer $m$ as in Lemma \ref{lemma.arith} such that
    $D=\left(\frac{L,mm_0}{\mathbb Q}\right)$ and an isogeny
    $\gamma\colon{E_{e}}\rightarrow {E_{\overline e}}$ 
    of degree $m$. Hence, there exists an isogeny 
    \begin{align*}
      \tilde{\psi_e}\colon A\rightarrow \tilde E^2
    \end{align*}
    of degree ${\rm deg}(\psi_e)\mid m\cdot(4c^2\cdot d)^2$, where
    $\tilde{\psi_e}:=(\gamma,{\rm id})\circ\psi_e$ and
    $\tilde E=E_{\overline e}$ is an elliptic curve with ${\rm
    End}(\tilde E)\supset \mathcal O_{L,c}$.
\end{enumerate}
\end{remark}
For the rest of the paragraph we fix the following notation. Let $A$ denote an abelian 
surface of type QM+CM over a field $F$. If ${\rm char}(F)$
is positive we also assume that $A$ is ordinary and hence, without loss of generality, that
$F$ is a finite field. Denote by $\psi\colon {\mathcal O}_D\rightarrow {\rm End}(A)$ its 
QM type. By Proposition \ref{pro:QMCM} $A$ is isogenous to a product $E^2$ of an elliptic 
curve of type CM by a field $L$,
which is isomorphic to the center of ${\rm End}^0(A)$. We fix an elliptic curve
in the isogeny class of $E$ with ${\rm End}(E)\simeq {\mathcal O}_{L,c}$, where
we denote by $c$ the conductor of the center of ${\rm End}(A)$ in its maximal
order ${\mathcal O}_{L}$. We call $c$ the \emph{central conductor} of $A$. By
abuse of notation, let 
$\psi\colon {\mathcal O}_D\otimes{\mathcal O}_{L,c}\rightarrow {\rm End}(A)$
also denote the prolongation of 
$\psi\colon {\mathcal O}_D\rightarrow {\rm End}(A)$ to ${\mathcal
O}_D\otimes{\mathcal O}_{L,c}$.\\
We use \cite{Kani:CMProd} to give a description of $A$ which is in some sense
dual to the description of Corollary \ref{case:ell}. 
\begin{theorem}\label{thm:Kani}
Let $A$ denote an abelian surface of type QM+CM over the field $F$, which is ordinary 
if ${\rm char}(F)>0$. Let $E$, $c$ be as above. Then there exists a non-trivial idempotent 
\[e_1\in\psi^{-1}({\rm End}(A))\in D\otimes_{\mathbb Q} L\]
and an elliptic curve $E'$ (unique up to isomorphism) such that 
\begin{align*}
 A\xrightarrow{(\psi(e_1),\psi(\overline{e_1}))}E\times E'
\end{align*}
defines an isomorphism of abelian surfaces, where $\overline{e_1}=1-e_1$.
Furthermore, if we denote by $c'$ the conductor of ${\rm End}(E')$ then $c'\mid
c$ holds.
\end{theorem}
\begin{proof}
  The abelian surface $A$ is isogenous to the product $E^2$ of the elliptic
curve $E$ of type CM. By \cite[Theorem 2]{Kani:CMProd} $A$ is isomorphic to the
product $E_1\times E_2$ of two isogenous elliptic curves of type CM (recall that
we assumed $A$ to be ordinary). The fact that we can choose $E_1$ to be $E$ as
above follows from the classification of products of elliptic curves in
\cite{Kani:CMProd} (see \cite[Theorem 67]{Kani:CMProd}). This determines $E':=E_2$
up to isomorphism. The fact that $c'\mid c$ easily follows as $c$ is the
conductor of the center of ${\rm End}(E\times E')$.
\end{proof}
\begin{remark}
 The situation for QM+CM abelian varieties $A$ of dimension $g\geq 4$ is more
difficult. Over ${\mathbb C}$ there exist abelian varieties $A$ of type
QM+CM with non-product varieties in its isogeny class (see \cite[Satz
0.1]{Schoen}).
\end{remark}
We denote for subsets $R,S\subset L$
\[ (R:S)_L:=\{x\in L\colon xS\subset R \} .\]
Given a product $E\times E'$ and an isogeny $\pi\colon E\rightarrow E'\simeq E/E[I]$ with
${\rm End}(E)\simeq \mathcal O_{L,c}$ and ${\rm End}(E')\simeq (I:I)_L \simeq \mathcal O_{L,c'}$
we can identify
\[{\rm End}(E\times E')\simeq\begin{pmatrix} \mathcal O_{L,c} & (I:\mathcal O_{L,c})_L\\%
			(\mathcal O_{L,c}:I)_L & \mathcal
O_{L,c'}\end{pmatrix}\]
via $\pi$.\par
Let $e$ denote the idempotent corresponding to the projection to
the first factor of 
$E\times E'$. Then $\psi\colon D\rightarrow {\rm End}^0(E\times E')$
satisfies $\psi(\mathcal O_D)\subset {\rm End}(E\times E')$ if and 
only if under this identification
\begin{align*}
 e\psi(\mathcal O_D)e&\subset \mathcal O_{L,c}\;e,& 
 e\psi(\mathcal O_D)\overline e &\subset (I:\mathcal O_{L,c})_L\;\tilde e,\\
 \overline e\psi(\mathcal O_D)e &\subset (\mathcal O_{L,c}:I)_L\;\tilde e^t, &
 \overline e\psi(\mathcal O_D)\overline e&\subset \mathcal O_{L,c'}\;\overline e,
\end{align*}
where $\tilde e=\left(\begin{smallmatrix}0&1\\0&0\end{smallmatrix}\right)$.\par
This gives rise to the following description of abelian surfaces of type QM+CM with
central conductor $c$ up to isomorphism.
\begin{corollary}[ {of Theorem \ref{thm:Kani}} ]\label{cor.of.Kani}
Denote $F$ an algebraically closed field of characteristic zero and $E$ a fixed 
elliptic curve $E$ with ${\rm End}(E)\simeq {\mathcal O}_{L,c}$. Let $\mathcal A_{D,c}$ be 
the set of abelian surfaces $A$ of type QM+CM over $F$ with central 
conductor $c$ up to isomorphism as abelian surfaces. Denote by $\mathcal X_{D,c}$ 
the set of elliptic curves $E'$ (up to isomorphism) with 
${\rm End}(E')\supset {\mathcal O}_{L,c}$, such that ${\mathcal O}_D$ embeds 
in ${\rm End}(E\times E')$. Then $A\simeq E\times E'
\mapsto E'$ defines a bijective mapping between $\mathcal A_{D,c}$ and
$\mathcal X_{D,c}$.
\end{corollary}
\begin{proof}
 Given an (ordinary) abelian variety $A$ of type QM+CM with central conductor $c$. Then
there exists an isomorphism $A\simeq E\times E'$ for some (ordinary) elliptic curve $E'$
isogenous to $E$ with ${\rm End}(E')\supset {\rm End}(E)$ by Theorem
\ref{thm:Kani}. It is injective by \cite[Theorem 67]{Kani:CMProd}.
\end{proof}
\begin{remark}
The statement of Theorem \ref{cor.of.Kani} holds true for $F$ an algebraically closed 
field of positive characteristic if we add the preposition ordinary to all abelian varieties 
occuring in the statement.
\end{remark}
In the following we define a action of the ideal class group ${\rm Id}({\mathcal O}_{L,c})/\simeq$ 
on $\mathcal X_{D,c}$, or equivalently on $\mathcal A_{D,c}$. For the convenience of the 
reader we repeat the following well-know fact, which can also be found in \cite{Kani:CMProd}.

\begin{proposition}[ {\cite[Cor.\ 21]{Kani:CMProd}} ]\label{kani:bijection}
  Let $E$ be a fixed CM elliptic curve with ${\rm End}(E)\simeq {\mathcal O}_{L,c}$. Denote by
  ${\rm Isog}^+(E)$ the set of elliptic curves isogenous to $E$ with ${\rm End}(E')\supset {\rm End}(E)$.
  Then the map
      \begin{align*}
	I_E^+\colon {\rm Id}({\mathcal O}_{L,c})/\simeq &\rightarrow {\rm Isog}^+(E)\\
	I'&\mapsto E':=E/E[I'].
      \end{align*} 
   defines a bijection of sets.
\end{proposition}
The bijection $I_E^+$ of Proposition \ref{kani:bijection} turns ${\rm Isog}^+(E)$ into a principal 
homogeneous space. We describe this more explicitly. Under the bijection $I_E^+$ we may view every 
$E'\in {\rm Isog}^+(E)$ as quotient $\pi\colon E\rightarrow E'=E/E[I']$. On the element $E'$ the action of 
${\rm Id}({\mathcal O}_{L,c})/\simeq$ is then given by $(J,E')\mapsto E/E[JI']$. Let 
$\tau\colon L\xrightarrow{\sim}{\rm End}^0(E)$ denote a fixed embedding. We furthermore fix for $E'$ 
the embedding $L\xrightarrow{\sim} {\rm End}^0(E')$ given by $\alpha\mapsto \pi^{-1}\circ \tau(\alpha) \circ \pi$ 
and denote ${\mathcal O}_{L,c'}$ the endomorphism ring ${\rm End}(E')$ under this identification.
Then the action of ${\rm Id}({\mathcal O}_{L,c})/\simeq$  on ${\rm Isog}^+(E)$ can simply be written 
(\cite[Prop.\ 12]{Kani:CMProd}) as
\[ (J,E')\mapsto E'/E'[{\mathcal O}_{L,c'}J]. \]
When we restrict the action to the Picard group 
\[I\in{\rm Pic}({\mathcal O}_{L,c})=\{I\in{\rm Id}(\mathcal O_{L,c})\colon (I:I)=\mathcal O_{L,c}\}/\simeq\]
then ${\rm End}(E/E[I])\simeq {\mathcal O}_{L,c}$
\begin{proposition}\label{prop.Kani}
We use the notations and assumptions of Corollary \ref{cor.of.Kani}.
\begin{enumerate}
 \item If $E''$ denotes an elliptic curve isogenous to $E$ with ${\rm
    End}(E')\subset{\rm End}(E'')$ for some $E'\in\mathcal X_{D,c}$ then 
    $E''\in\mathcal X_{D,c}$ holds.
 \item Denote $E'\in\mathcal X_{D,c}$. Then $\mathcal X_{D,c'}={\rm Isog}^+(E)$ holds.
\end{enumerate}%
\end{proposition}
\begin{proof}
 The first claim follows immediately from the fact that ${\rm
  End}(E\times E')\subset{\rm  End}(E\times E'')$. The second claim follows analogously.qed

\end{proof}
\begin{remark}
 Assume we know the set $\Gamma\subset\mathbb N$ of numbers $c'$ such that
  there exists  $E'\in\mathcal X_{D,c}$ with ${\rm End}(E')\simeq\mathcal
  O_{L,c'}$ or equivalently the maximal elements of $\Gamma$ (under the  partial
  ordering given by  divisibility). One can calculate the cardinality of
  ${\mathcal A}_{D,c}$ using Proposition \ref{prop.Kani}.
\end{remark}
\begin{proposition}
 The group ${\rm Id}({\mathcal O}_{L,c})/\simeq$ acts on $\mathcal A_{D,c}$.
 This action restricted to the Picard group ${\rm Pic}({\mathcal O}_{L,c})$ is given by
  \[([I],E\times E')\mapsto  E/E[I]\times E'\simeq  E\times   E'/E'[\mathcal
  O_{L,c'} I].\] 
\end{proposition}
\begin{proof}
  The action is given via the isomorphism $\mathcal A_{D,c}\simeq\mathcal X_{D,c}$ and 
  Lemma \ref{prop.Kani}.\\
  By \cite[Prop.\ 65]{Kani:CMProd} there exists an isomorphism
  \[E/E[I]\times E'\simeq  E\times  E'/E'[\mathcal  O_{L,c'} I]\]
  if and only if $I I'\simeq {\mathcal O}_{L,c'} I I'$, where $I'$ is an
  ideal such that $E'\simeq E/E[I']$ holds. This is obviously satisfied.
\end{proof}
Assume we are given an abelian variety $A\simeq E\times E'$ of QM-type $\psi_{\rm QM}\colon
\mathcal O_D\hookrightarrow {\rm End}(E\times E')$. In this context we give
the isogeny $\psi_e\colon A\rightarrow E_e\times E_{\overline e}$ more
explicitly. In the following we always identify $D\otimes_{\mathbb Q} L$ with
${\rm End}^0(E\times E')$ via $\psi_{\rm QM}$. Hence we can write
$e=x_1\otimes 1+ x_2\otimes \alpha\in D\otimes_{\mathbb Q} L$ for
the idempotent corresponding to the projection to the first factor. The
following identities hold:
\begin{align}
 {\rm tr}_{D/\mathbb Q}(x_1)&=1, & {\rm tr}_{D/\mathbb Q}(x_2)&=0,\\
  {\rm tr}_{D/\mathbb Q}(x_1 x_2^{-1})&=x_1x_2^{-1}-x_2^{-1}\overline{x_1}=0,
\intertext{ and, if we denote $-d=(x_1x_2^{-1})^2$,}
 {\rm n}_{D/\mathbb Q}(x_1)&=d\cdot{\rm n}_{D/\mathbb Q}(x_2).
\end{align}
Let $\iota\colon {\mathcal O}_{L,c}\rightarrow
{\mathcal O}_D$ denote the embedding corresponding to $e$ of Proposition \ref{idem},
ie.\ $\iota(\alpha)=x_1x_2^{-1}$. In order to study
$\psi_e=(e\iota(c\alpha),\overline{e c\iota(\alpha)})=(1\otimes
c\alpha,\iota(\alpha)\otimes 1)$ of Theorem \ref{haupt} we study the projection of
$\iota(\alpha)\otimes 1$ to the first and second factor of $E\times E'$.
\begin{lemma}\label{ident.iota}
 The following identities hold.
\begin{align}
 \overline e\iota(\alpha)&=\overline e\iota(\alpha)\overline e=\overline e\cdot
(-1\otimes\alpha),\\
 e\iota(\alpha)e&=e\cdot(1\otimes\alpha),\\
 e\iota(\alpha)\overline e&=\iota(\alpha)(\overline{x_1}-{x_1})\otimes
1+(\overline{x_1}-{x_1})\otimes\alpha,\\
 &=(e-\overline e)\iota(\alpha)-1\otimes\alpha.
\end{align} 
\end{lemma}
\begin{proof}
We calculate:
\begin{align*}
 \overline e\iota(\alpha)e&=(\overline x_1\otimes 1-x_2\otimes \alpha)\cdot
(x_1 x_2^{-1})\cdot(x_1\otimes 1+x_2\otimes \alpha)\\
 &=({\rm n}_{D/\mathbb Q}(x_1) x_2^{-1} x_1+dx_2x_1)\otimes 1
  +({\rm n}_{D/\mathbb Q}(x_1)-x_2x_1x_2^{-1}x_1)\otimes \alpha\\
 &=0,\\
 \overline e\iota(\alpha)&=\overline e\iota(\alpha)\overline e
  =(\overline x_1\otimes 1-x_2\otimes \alpha)\cdot
(x_1 x_2^{-1})\cdot(\overline{x_1}\otimes 1-x_2\otimes \alpha)\\
 &=({\rm n}_{D/\mathbb Q}(x_1) x_2^{-1} \overline{x_1}-dx_2x_1)\otimes 1
  +(-{\rm n}_{D/\mathbb Q}(x_1)-x_2x_1x_2^{-1}\overline{x_1})\otimes
\alpha\\
 &=-d x_2\otimes 1-\overline{x_1}\otimes \alpha
 =\overline e\cdot (-1\otimes\alpha).
\end{align*}
Analogously,
\[
e\iota(\alpha)e=-dx_2\otimes 1+x_1\otimes\alpha=e\cdot(1\otimes\alpha).
\]
Furthermore,
\begin{align*}
 e\iota(\alpha)\overline e&=(x_1\otimes 1+x_2\otimes \alpha) \cdot
(x_1 x_2^{-1})\cdot(\overline x_1\otimes 1-x_2\otimes \alpha)\\
 &=(x_1^2x_2^{-1}\overline{x_1}+dx_2x_1)\otimes 1
  +(-x_1^2+x_2x_1x_2^{-1}\overline{x_1})\otimes \alpha,\\
 &=(x_1x_2^{-1}\cdot(1-2x_1))\otimes 1+(1-2x_1)\otimes\alpha\\
 &=\iota(\alpha)(\overline{x_1}-{x_1})\otimes
1+(\overline{x_1}-{x_1})\otimes\alpha.
\end{align*}
We calculate
\begin{align*}
 (e-\overline e)\iota(\alpha)&=e\iota(\alpha)e+e\iota(\alpha)\overline
e-\overline e\iota(\alpha)\overline e\\
&=e(1\otimes\alpha)+e\iota(\alpha)\overline e-\overline e(-1\otimes \alpha)\\
&=1\otimes\alpha+e\iota(\alpha)\overline e.
\end{align*}
\end{proof}
Using Lemma \ref{ident.iota} and Theorem \ref{thm:Kani} we can give $\psi_e\colon
A\rightarrow E_e\times E_{\overline e}$ of Theorem \ref{haupt} more explicitly.
\begin{corollary}[ {of Theorem \ref{haupt}} ]\label{can.psi}
 In the notation of Corollary \ref{case:ell}. Assume $A\simeq E\times E'$. Then the
isogeny 
\[ \psi_e\colon A\rightarrow E_e\times E_{\overline e} \]
with ${\rm deg}(\psi_e)\mid (4c^2\cdot d)^2$ which is asserted to exist in Corollary
\ref{case:ell} can be given as (projection onto the image) of
\[ (c\alpha|_E,\gamma,c\alpha|_{E'})\colon E\times E'\rightarrow
E\times E'\times E',\]
where $\gamma$ is the morphism $\gamma\colon E\rightarrow E'$ induced by
$\psi_{\rm QM}(e\iota(c\alpha)\overline e)\in {\rm End}(E\times E')$.
Furthermore, we conclude that $E_e\simeq{\rm Img}(c\alpha|_E,\gamma)(E)$ and
$E_{\overline e}\simeq E'$.
\end{corollary}
\begin{proof}
The isogeny in Corollary \ref{case:ell} is induced by the ${\rm End}(A)$-ideal
\begin{align*}
I&:=(\psi_{\rm QM}(e\iota(c\alpha)),\psi_{\rm QM}
(\overline{e\iota(c\alpha)}))\\
&=(\psi_{\rm QM}(1\otimes c\alpha),
\psi_{\rm QM}(\iota(c\alpha)))\\
&=(c\alpha|_E,c\alpha|_{E'}, \psi_{\rm QM}(e\iota(\alpha)\overline e))
\end{align*}
by the identities of Lemma \ref{ident.iota}. The statement follows.
\end{proof}
Finally we are interested in polarizations on QM+CM abelian surfaces. Therefore, 
let $A$ be an abelian surface of type QM+CM and $\lambda\colon A\rightarrow A^t$ 
a principal polarization. By Theorem \ref{thm:Kani} there exists an isomorphism 
$\psi\colon A\rightarrow E\times E'$, where $E,E'$ are isogenous elliptic curves. 
Hence the polarization $\lambda$ is given by
\begin{align*}
 \lambda\colon A\xrightarrow{\psi} E\times E' \xrightarrow{\tilde\lambda} E^t\times (E')^t
\xrightarrow{\psi^t} A^t,
\end{align*}
for some polarization $\tilde\lambda$ on $E_1\times E_2$. Or, put differently, we can factor
$\lambda$ as a composite
\begin{align*}
 A \xrightarrow{\beta} A\xrightarrow{\psi} E\times E'  \xrightarrow{\lambda_{\rm pol}} E^t\times (E')^t
 \xrightarrow{\psi^t} A^t \xrightarrow{\beta^t} A^t,
\end{align*}
where the middle arrow is the product polarization on $E\times E'$ and 
$\beta\colon A\rightarrow A$ is an isomorphism.\par
Theorem \ref{thm:Kani} and the above remark give an alternative approach to
\cite{BG} for describing CM points of Shimura curves corresponding to 
the quaternion algebra $D$. Furthermore this approach is not restricted to abelian surfaces in
characteristic $0$.\par
One application of Theorem \ref{thm:Kani} is given in \cite{Diss}. In the
following we
explain roughly the idea. In \cite{Diss} we study Shimura curves $C$ describing
principally polarized
abelian surfaces of type QM. Denote by $x\in C$ a CM point and by $A$ the
corresponding principally polarized abelian
surface. As $x$ is a CM point the variety $A$ is defined over a number field. We
assume that $A$ has good ordinary
reduction $\overline A$ at a prime $p\in{\mathbb N}$. We are interested in the
locus of the Shimura curve in the
formal deformation space $\mathcal M$ of $\overline A$ and in the CM points in
that locus. We denote by $F\colon \overline
A\rightarrow \overline A$ the absolute Frobenius on $\overline A$. The geometry
of the CM points $y\in C$
in $\mathcal M$ is controlled by the natural number $n$ such that $[p^n]F$ lifts
to an endomorphism of $A_y$. In
\cite[Prop.\ 2.6.7]{Diss} we show, using Theorem \ref{thm:Kani}, that lifting
$[p^n]F$ to an endomorphism on $A_y$ is
equivalent to $[{\mathcal Z}(\overline A):{\mathcal Z}(A_y)]=p^n$, where 
$\mathcal Z(\overline A)$ (resp.\ ${\mathcal Z}(A_y)$) 
denotes the center of ${\rm End}(\overline A)$ (resp.\ ${\rm End}(A_y)$).

\section*{Acknowledgments}
 This work was supported by the \emph{Ministerium f\"ur Wissenschaft, Forschung und Kunst} of the state of
Baden-W\"urttemberg (Germany) and the European Research Council [ERC-2010-StG to M.\ M\"oller].\\
I want to thank Irene I.\ Bouw for many helpful comments on this work.

\end{document}